\documentclass[12pt, reqno]{amsart}
\usepackage{amsmath, amsthm, amscd, amssymb, graphicx, color}
\usepackage{lmodern}
\usepackage{mathrsfs}
\usepackage{array}
\usepackage{tabularx}
\usepackage[colorlinks=true, linkcolor=blue, citecolor=blue, urlcolor=blue]{hyperref}
\usepackage{enumitem} 
\newtheorem{theorem}{Theorem}[section]
\newtheorem{lemma}[theorem]{Lemma}

\newtheorem{corollary}[theorem]{Corollary}
\theoremstyle{definition}
\newtheorem{definition}[theorem]{Definition}

\theoremstyle{remark}
\newtheorem{remark}[theorem]{Remark}
\numberwithin{equation}{section}
\newcommand{\tp}{\mathcal{T}_\phi}
\newcommand{\hp}{\mathcal{H}_\phi}

\newcommand{\h}{\mathcal{H}}

\newcommand{\f}{\mathcal{J}}

\newcommand{\1}{\theta H^2}
\newcommand{\2}{K_\theta}
\newcommand{\ch}{\widehat{H}}
\newcommand{\ho}{\overline{H_0^2}}
\newcommand{\pt}{P_\theta}
\newcommand{\li}{L^\infty}
\usepackage{tikz-cd}
\renewcommand{\arraystretch}{1.4}

\begin{document}
	
	\title[Restricted Toeplitz and Hankel Operators]{Restricted Toeplitz and Hankel Operators}
	
	\author[Priyanka Aroda, Arup Chattopadhyay, and Supratim Jana]{Priyanka Aroda$^*$, Arup Chattopadhyay$^{**}$ and Supratim Jana$^{***}$}
	\maketitle
	
	\paragraph{\textbf{Abstract}}
	We introduce and systematically study a class of operators that arise naturally due to the Beurling decomposition of the Hardy space $H^2=\1 \oplus \2$. While the compressions of classical Toeplitz and Hankel operators to the Beurling subspace $\1$ and the model space $\2$ account for the diagonal components of the decomposition, the corresponding off-diagonal operators have remained largely unexplored. Motivated by this, we introduce and analyze a new class of operators, termed \emph{restricted Toeplitz} and \emph{restricted Hankel operators}, acting between Beurling subspace $\eta H^2$ and model space $\2$.

	Within this framework, we obtain necessary and sufficient conditions for the vanishing, finite-rank, and compactness properties of these operators. We further establish algebraic characterizations in the spirit of Brown–Halmos \cite{BH} and Sarason \cite{SAR, DES}, showing that these operators can be identified through certain operator equations involving compressed shifts. As an application, we introduce the notions of small and big truncated Toeplitz operators, and provide criteria for when they vanish, have finite rank, or are compact.

	\vspace{0.5cm}
	
	\vspace{0.5cm}
	
	\tableofcontents	
	
	\section{{Introduction and Preliminaries}}
	Let $\mathbb{D}:=\{ z \in \mathbb{C}: |z| < 1 \}$ denote the unit disk in the complex plane $\mathbb{C}$ and $\mathbb{T}:= \{ z \in \mathbb{C}: |z|=1 \}$ its boundary. Also, let $L^2$ denote the space of square integrable functions on the unit circle $\mathbb{T}$ with respect to the normalized Lebesgue measure. And by $L^\infty$, we denote the von Neumann algebra of essentially bounded Lebesgue measurable functions on $\mathbb{T}.$ 
	
	The \emph{Hardy-Hilbert space} $H^2(\mathbb{D})$ is the space of holomorphic functions on $\mathbb{D}$ defined by
	$$ H^2(\mathbb{D}):= \left\{ f(z) = \sum_{n=0}^\infty a_n z^n \;: \; \sum_{n=0}^\infty |a_n|^2 < \infty \right\}.$$ 
	The Banach algebra of bounded analytic functions on $\mathbb{D}$ is denoted as $H^\infty(\mathbb{D})$. By the elegance of Fatou's theorem \cite{NKN, VVP}, the space $H^2(\mathbb{D})$ is identified (through boundary limit) as a closed subspace of $L^2$, whose every element only has non-negative Fourier coefficients, and it is denoted by $H^2(\mathbb{T})$. For convenience, we write $H^2$ and $H^\infty(:= \li \cap H^2)$ irrespective of $\mathbb{D}$ and $\mathbb{T}$, according to the context. An inner function $\theta$ is an $H^\infty$ element whose boundary value is unimodular almost everywhere over $\mathbb{T}.$


	In the mid 1960's, Brown-Halmos introduced the Toeplitz operator $T_\phi$ in their seminal paper \cite{BH}, and they defined this as:
	
	$$ T_\phi: H^2 \rightarrow H^2 \text{ by } T_\phi (f)=P(\phi f) \quad \text{ for $\phi\in L^\infty$, } $$
	where $P: L^2 \rightarrow H^2$ is the orthogonal projection (also known as the Szeg\"o projection). And the corresponding Hankel operator $H_\phi$ is defined as:
	\begin{equation}\label{han}
		H_\phi:H^2 \rightarrow \overline{H_0^2}\left(:= ({H^2})^\perp \right) \text{ by } H_\phi (f)=(I-P)(\phi f),
	\end{equation}
	where $H_0^2= z H^2$ and $\overline{H_0^2}$ stands for the complex conjugate of the elements of $ H_0^2$.
	
	In the literature, for the sake of studying several properties (like normality and related properties), the Hankel operator is also defined as
	$$ H_\phi: H^2\rightarrow H^2 \text{ by } H_\phi (f)= \mathcal{J}(I-P)(\phi f)= P\mathcal{J}(\phi f),$$ where $\f:L^2 \rightarrow L^2$ is a unitary involution, known as the \emph{flip} operator, defined by $(\mathcal{J}f)(\xi) = \bar{\xi}f(\bar{\xi}) \text{ for } \xi\in \mathbb{T}$. To distinguish, we use the notation $\widehat{H}_{\phi}$ for the Hankel operator with co-domain in $\overline{H_0^2}$ (see~\eqref{han}). We also sometimes use the notation $Q$ for $I-P$.
	\vspace{0.1in}
	
	The most elementary Toeplitz operator that holds a significant importance in the study of Hardy spaces is $T_z$, which is known as the (forward) shift operator $S$. The Beurling's theorem \cite{AB}, a cornerstone of function theory, provides a complete characterization of all $S$-invariant subspaces of $H^2$. And they are precisely of the form $\1$, where $\theta$ is an inner function. It follows that the invariant subspace under the action of $S^*$ (the backward shift operator) is $({\theta H^2})^\perp$ in $H^2,$ and they are known as the Model spaces, denoted by $K_\theta$, and determined by $H^2\cap\theta~\ho$.
	\vspace{0.1in}
	
	The theory of Toeplitz and Hankel operators on $H^2$ has long been a center of attraction to operator theory and complex analysis. Foundational work by Brown--Halmos, Sarason, Nehari, and many others has led to rich algebraic and analytic characterizations of these operators along with their symbols. Over the years, theories have been developed in several directions, one of the most significant being the study of compressed operators acting on invariant and coinvariant subspaces of $H^2$. 
	
	In 2007, Sarason in \cite{DES} introduced the notion of the truncated Toeplitz operators, which are essentially compressions of the classical Toeplitz operators to the model spaces $\2$. It is presented as a densely defined operator given by
	$$A_\phi^\theta:\2\rightarrow\2 \quad \text{ as }\quad A_\phi^\theta(h) = P_\theta(\phi h),$$
	where $\phi \in L^2(\mathbb{T})$, $h \in K_\theta$ with $\phi h \in L^2(\mathbb{T})$, and $P_\theta : L^2(\mathbb{T}) \to K_\theta$ denotes the orthogonal projection. In particular, $A_z^\theta =\pt S|_{\2}(:=S_\theta)$, and it is known as the compressed shift operator. Since its introduction, this framework has been widely adopted to develop a diverse range of theories. 
	\vspace{0.1in}
	
	At the beginning of the last decade, Gu introduced the densely defined truncated Hankel operators (THO) on the model space $\2$ in a preprint as 
	$$B_\phi^\theta:\2\rightarrow\2 ~\text{ by }~ B_\phi^\theta(h) = P_\theta\mathcal{J}(\phi h), ~ \phi \in L^2, \; h \in K_\theta \cap L^\infty,$$
	where $\mathcal{J}$ is the flip operator. And recently, this work was revised and published jointly by Gu-Ma in \cite{GM}.
	\vspace{0.1in}

	Considering the decomposition $H^2=\1\oplus \2$, we present the Toeplitz and the Hankel operators in the following block-matrix form:
	$$ \scalebox{1.1}{$T_\phi :=
		\begin{array}{cc|c}
			\theta H^2 & K_\theta & \\ \cline{1-2}
			\widetilde{T}_\phi & \widehat{X} & \theta H^2 \\
			X & A_\phi^\theta & K_\theta
		\end{array}$}
	~\text{ , and } ~ 
	\scalebox{1.1}{$H_\phi :=
		\begin{array}{cc|c}
			\theta H^2 & K_\theta \\ \cline{1-2}
			\widetilde{H}_\phi & \widehat{Y} & \theta H^2 \\
			Y & B_\phi^\theta & K_\theta 
		\end{array}.$} $$
	\vspace{0.1in}
	
	The operators $\widetilde{T}_\phi, \widetilde{H}_\phi:\1\rightarrow\1$ in the $(1,1)$ positions are the compressions of the Toeplitz and the Hankel operators to the Beurling subspace $\1$. And their actions are exactly similar to classical Toeplitz and Hankel operators because $\1$ is isometric to $H^2$, and the operator $\mathcal{S}_{(\theta)}:= P_{\1}S|_{\1}$ is similar to $S$. 
	
	Whereas, the operators appearing in the $(2,2)$ positions are the truncated Toeplitz and the truncated Hankel operators. These operators, however, exhibit markedly different structures and behaviors. A substantial body of versatile theory of these truncated operators has been developed recently by various mathematicians, such as the study of symbols (\cite{BCFMT, BBK}), spatial isomorphism and unitary equivalence (\cite{CGRW, GRW1, GRW2, GPR}), compactness (\cite{GM, MZ1}), etc. 
	\vspace{0.1in}
	
	Therefore, it is evident that the study of both diagonal operators is well-established and has been ongoing. In contrast, the off-diagonal (or cross-diagonal) operators $X$, $\widehat{X}$, $Y$, and $\widehat{Y}$ have not been systematically investigated. The purpose of this article is to identify and study these less-explored operators. To address this, we introduce the following two operators.
	
	\begin{definition}\label{def}
		Given two inner functions $\eta$ and $\theta$, and $\phi\in \li$, the restricted Toeplitz operator (RTO) is defined as: 
		\begin{equation} \label{rto}
			\tp: \eta H^2 \rightarrow \2  \quad \text{ by } \quad \tp (h)= \pt(\phi h),
		\end{equation} and the restricted Hankel operator (RHO) is defined as: 
		\begin{equation}\label{rho}
			\hp: \eta H^2 \rightarrow \2 \quad \text{ by } \quad \hp (h)= \pt\f(\phi h).
		\end{equation}
	\end{definition}
	
	Observe that both the domain and range spaces in this context are precisely the subspaces characterized by Beurling's theorem—one being invariant under $S$, the other invariant under $S^*$. Moreover, when $\eta=\theta$, they occur in the $(2,1)$ positions of the above block matrices.
	\vspace{0.1in}
	
	Also, to address the complementary off-diagonal operators $\widehat{X}, \widehat{Y}$ in the $(1,2)$ positions, we consider the following two operators, which will play a useful role in the upcoming sections.
	\begin{equation}\label{1}
		\tau_\phi:\2 \rightarrow \eta H^2 \text{ by } \tau_\phi(h)= P_{\eta H^2} (\phi h),  \quad \text{for $\phi \in \li$ }
	\end{equation}
	\text{ and }
	\begin{equation}\label{2}
		h_\phi:\2 \rightarrow \eta H^2 \text{ by } h_\phi (h) = P_{\eta H^2} \f (\phi h), \quad \text{for $\phi \in \li$. }
	\end{equation}
	These two operators appear to be the adjoints of RTO and RHO, respectively, via 
	\begin{equation}\label{adj}
		\tau_\phi = \mathcal{T}_{\bar{\phi}}^* \quad \text{ and } \quad h_\phi = \mathcal{H}_{{\phi}^*}^*,
	\end{equation}
	where $\phi^*(\xi)=\overline{\phi(\overline{\xi})}$ for $\phi\in L^2$. 
	\vspace{0.1in}

	In our first main result, we study some analytic properties of restricted Toeplitz and Hankel operators. Inspired by the work of Ma-Yan-Zheng \cite{MYZ, MZ1}, we establish necessary and sufficient conditions for these operators to be zero, finite rank, and compact. These characterizations depend on the interaction between the symbol $\phi$ and the inner functions $\eta$ and $\theta$. In particular, the compactness involves the behavior of these functions on support sets associated with the maximal ideal space of $L^\infty$.
	\vspace{0.1in}
	
	In addition to these analytic properties, we establish algebraic characterizations of restricted operators parallel to classical ones. By carefully analyzing the algebraic interplay, we obtain characterizations formulated in terms of operator equations. The following is the characterization theorem for the restricted Toeplitz operator (RTO). 
	
	\begin{theorem}\label{3th1}
		Let $\eta, \theta$ be two inner functions, and let $X:\eta H^2\rightarrow\2$ be a bounded linear operator. Then, $X$ is a restricted Toeplitz operator if and only if 
		$$  X- S_\theta^* X \mathcal{S}_{(\eta)} = S^* (\theta) \otimes P(\theta \overline{z\phi\eta}), $$
		for some $\phi\in L^{\infty}$. Moreover, $X=\tp$.  
	\end{theorem}
	
	The corresponding characterization for restricted Hankel operators is as follows.
	
	\begin{theorem}\label{3th2}
		Let $\eta, \theta$ be two inner functions and let $X:\eta H^2\rightarrow\2$ be a bounded linear operator. Then, $X$ is a restricted Hankel operator if and only if
		\begin{equation*}\label{3eq5}
			X\mathcal{S}_{(\eta)} - S_\theta^*X = S^*({\theta}) \otimes P(\breve{\theta}~{\overline{z\eta\phi}}),
		\end{equation*}
		for some $\phi\in L^{\infty}$.
	\end{theorem}
	
	Moreover, an application of the commutant lifting theorem, we obtain a complete structural description of a certain class of operators $A$ intertwining $\mathcal{S}_{(\eta)}$ and $S_\theta$ as follows.
	
	\begin{theorem}\label{3thm}
		Let $\eta$ and $\theta$ be two inner functions and let $A: \eta H^2\rightarrow\2$ be a bounded operator. Then $ A\mathcal{S}_{(\eta)}=S_\theta A$ if and only if $A= \tp$ with $\phi \in \bar{\eta}H^\infty$.
	\end{theorem}

	Turning to the truncated operators, Bessonov in \cite{B1} (2015) introduced another notion of truncated Hankel operators (densely defined) and defined them as: 
	$$ \Gamma_\phi: \2 \rightarrow \overline{z \2} ~\text{ by }~ \Gamma_\phi(f) = P_{\bar{\theta}}(\phi f), f\in \2 \cap L^\infty,$$
	where $P_{\bar{\theta}}: L^2 \rightarrow \overline{z \2}$ is the orthogonal projection. This operator has been termed the little truncated Hankel operator by Ma-Yan-Zheng in \cite{MYZ} (2018) while studying the compactness of the big truncated Hankel operators (BTHO), defined as follows: 
	$$ H_\phi^\theta: K_\theta \rightarrow (\2)^\perp ~\text{ by }~ H_\phi^\theta (f) = ( I-\pt ) (\phi f), f\in \2 \cap L^\infty.$$
	Motivated by these structural considerations and the developments, we introduce the following two operators.
	
	\begin{definition}\label{Def}
		The small truncated Toeplitz operator (STTO) on the model space $\2,$ is defined as $$ t_\phi^\theta:\2 \rightarrow (\overline{z\2})^\perp \left( = \bar{\theta} \ho \oplus H^2 \right) \text{ by } t_\phi^\theta(f)= (I-P_{\bar{\theta}})(\phi f), \phi \in \li,$$ and the big truncated Toeplitz operator (BTTO) as 
		$$ T_\phi^\theta: \2 \rightarrow (\2)^\perp \left( = \ho \oplus \theta H^2 \right) \text{ by } T_\phi^\theta(f) = (I-\pt)\f (\phi f), \phi \in \li.$$
	\end{definition}
	As in the case of the BTHO, the STTO $t_\phi^\theta$ and BTTO $T_\phi^\theta$ are bounded only when the associated symbol $\phi$ is essentially bounded.
	
	The following table summarizes the definitions of truncated operators discussed so far.
	\vspace{0.1in} 
	
	\setlength{\arrayrulewidth}{0.1mm}
	\setlength{\tabcolsep}{5pt}
	\renewcommand{\arraystretch}{1.3}
	
	\begin{tabular}{|p{6cm}|p{7cm}|}
		\hline
		\multicolumn{2}{|c|}{\textbf{Operators Summary}} \\
		\hline
		\textbf{Main Operator} & \textbf{Compression/Restriction} \\
		\hline
		$M_\phi : L^2 \to L^2$ & $T_\phi : H^2 \to H^2,\; T_\phi = P M_\phi |_{H^2},$ \cite{BH} \\
		\hline
		$T_\phi : H^2 \to H^2,\; T_\phi = P M_\phi |_{H^2}$  & 
		$A_\phi^\theta : \2 \to \2;\; A_\phi^\theta = P_\theta M_\phi |_{\2},$ \cite{DES} \\
		\hline
		$T_\phi: H^2 \to (\overline{zH^2})^\perp;\;$ \newline$ T_\phi= (I - Q) M_\phi |_{H^2}$& $ t_\phi^\theta: \2 \rightarrow(\overline{z\2})^\perp; t_\phi^\theta= (I-P_{\bar{\theta}})M_\phi|_{\2} $, \newline STTO in \eqref{Def}. \\
		\hline
		$\ch_\phi : H^2 \to \overline{zH^2};\; \ch = Q M_\phi |_{H^2}$ & 
		$\Gamma_\phi^\theta : \2 \to \overline{z\2};\; \Gamma_\phi^\theta = P_{\bar{\theta}} M_\phi |_{\2},$ \cite{B2} \\
		\hline
		$\ch_\phi : H^2 \to (H^2)^\perp;\;$ \newline$ \ch_\phi= (I - P) M_\phi |_{H^2}$ & 
		$H_\phi^\theta : \2 \to (\2)^\perp;\; \newline H_\phi^\theta = (I - P_\theta) M_\phi |_{\2},$ \cite{MYZ} \\
		\hline
		$\f M_\phi : L^2 \to L^2$ & 
		$H_\phi : H^2 \to H^2;\; H_\phi = P \f M_\phi|_{H^2}$ \\
		\hline
		$H_\phi : H^2 \to H^2;\; H_\phi = P \f{M}_\phi |_{H^2}$ & 
		$B_\phi^\theta : \2 \to \2;\; B_\phi^\theta = P_\theta \f{M}_\phi |_{\2}$ \cite{GM} \\
		\hline
		$ \mathbb{T}_\phi: H^2 \rightarrow (H^2)^\perp;$\newline$ \mathbb{T}_\phi= (I-P)\f M_\phi|_{H^2} $ \newline acts like a Toeplitz operator \cite{CJ} & $ T_\phi^\theta:\2 \rightarrow(\2)^\perp;$\newline$ T_\phi^\theta= (I-\pt)\f M_\phi|_{\2}$, \newline BTTO in \eqref{Def}.  \\
		\hline
		
	\end{tabular}
	\vspace{0.1in}
	
	Note that in the above table, $P$ is replaced by $\pt$ whereas $Q$ is replaced by $P_{\bar{\theta}}$, and $H^2$ by $\2$ appropriately. The compressions and the restrictions are taken accordingly.
	
	\vspace{0.1in}
	
	The article is organized in the following manner. The study of the compactness of all kinds (zeroness, finiteness of rank, and compactness) of the restricted Toeplitz and Hankel operators is obtained in Section 2 (see Theorems~\ref{2t1}, \ref{2t2}, \ref{2t3}). Section 3 presents characterizations of RTO and RHO in terms of operator identities. Section 4 applies the restricted framework to small and big truncated Toeplitz operators. Section 5 contains special and concluding remarks.

	\section{Compactness of RTOs and RHOs}
	
	In this section, we provide the necessary and sufficient conditions for the restricted Toeplitz (RTO) and the restricted Hankel operators (RHO) to be zero, finite rank, and compact. We used the remarkable results obtained by Axler-Chang-Sarason in 1978 \cite{ACS}. Before proceeding further, let us start with the relation between the classical Toeplitz and Hankel operators.
	
	For symbols $\phi,\psi \in \li,$ the Toeplitz operator $T_\phi$ and the Hankel operator $\ch_\psi,$ satisfying the following relation
	\begin{equation}\label{relation1}
		T_{\phi\psi} - T_\phi T_\psi = \ch_{\bar{\phi}}^* \ch_{\psi}.
	\end{equation}
	Since, $\f P\f = Q$, we have $H_{\phi} = \mathcal{J}\, \widehat{H}_{\phi}$, therefore the above equation \eqref{relation1} can be expressed as
	\begin{equation}\label{relation2}
		T_{\phi\psi} - T_\phi T_\psi = \ch_{\bar{\phi}}^* \ch_{\psi} = \ch_{\bar{\phi}}^* \f \f \ch_{\psi}= (\f \ch_{\bar{\phi}})^*H_\psi= H_{\bar{\phi}}^*H_\psi = H_{\breve{\phi}}H_\psi,
	\end{equation} 
	where $\breve{f}(\xi)=f(\bar{\xi})$ for $f$ being in $L^2$.
	
	Now, it is well known that the orthogonal projection $\pt$ can be expressed as $$\pt=P-M_\theta P M_{\bar{\theta}}.$$ So, restricting $\pt$ over $H^2$ yields $ \pt|_{H^2}= I- T_\theta T_{\bar{\theta}}$ (as done in \cite{MYZ}, also see \cite{CL}). Thus, 
	\begin{equation}\label{2eq}
		\pt|_{H^2}= T_{\theta\bar{\theta}} - T_\theta T_{\bar{\theta}}=H_{\breve{\theta}} H_{\bar{\theta}}=\ch_{\bar{\theta}}^*\ch_{\bar{\theta}} \quad \text{ (using\eqref{relation1} and \eqref{relation2}) }.
	\end{equation}

	The following lemma expresses the RTO as a multiplication of two Hankel operators.
	
	\begin{lemma}\label{2l1}
		Let $\eta,\theta$ be two inner functions and consider the associated restricted Toeplitz operator (RTO) $\tp:\eta H^2 \rightarrow\2$ corresponding to $\phi\in\li$. Then, the RTO $\tp$ can be identified as the action of $H_{\breve{\theta}}H_{\phi\eta\bar{\theta}}$ on $H^2.$ 
	\end{lemma}
	
	\begin{proof}
		Let $h=\eta f \in \eta H^2$, where $f \in H^2$. By Definition~\ref{def},
		\begin{align*}
			\tp(h) = \pt(\phi h)  & = (P-M_\theta P M_{\bar{\theta}})(\phi h)\\
			& = P(\phi h) - M_\theta P( \bar{\theta} \phi h)\\
			& = P(\phi \eta f) - M_\theta P( \bar{\theta} \phi\eta f)\\
			& = T_{\phi \eta} (f) - T_\theta T_{\bar{\theta} \phi\eta}(f)\\
		\end{align*}
		
		Since $\theta\bar{\theta}=1$, we may write $T_{\phi\eta}(f)=T_{\theta\bar{\theta}\phi\eta}(f),$ therefore
		$$ \tp(h) = \bigl(T_{\theta\bar{\theta}\phi\eta}-T_\theta T_{\bar{\theta}\phi\eta}\bigr)(f).$$
		Now, by \eqref{relation2}, we have $$ T_{\theta\bar{\theta}\phi\eta}-T_\theta T_{\bar{\theta}\phi\eta}
		= H_{\breve{\theta}}H_{\bar{\theta}\phi\eta}.$$
		
		Thus, the operator $\tp$ on $\eta H^2$ is equivalent to the operator $H_{\breve{\theta}}H_{\phi\eta\bar{\theta}}$ on $H^2.$
	\end{proof}
	Therefore, the vanishing, finite rank, and compactness properties of $\tp$ are determined by the consequent properties of $H_{\breve{\theta}}H_{\bar{\theta}\phi\eta}$ on $H^2$, respectively.

	For the study of compactness of the restricted Hankel operator $\hp$, we use the standard fact that a bounded operator $A$ on a Hilbert space is zero, of finite rank, or compact if and only if $A^*A$ has the corresponding property.
	
	\begin{lemma}\label{2l3}
		The restricted Hankel operator $\hp$ is zero, finite rank, and compact if and only if $H_{\bar{\theta}}H_{\phi\eta}$ on $H^2$ is zero, finite rank, and compact.
	\end{lemma}
	
	\begin{proof}
		Let $h=\eta f \in \eta H^2$, where $f \in H^2$. By Definition~\ref{def}, $$ \hp(h)=\pt\f(\phi h)= \pt P \f(\phi \eta f) = \pt H_{\phi\eta}(f) .$$ Thus, $\hp$ on $\eta H^2$ is represented as $ \pt H_{\phi\eta}$ on $H^2$. Now,
		\begin{align*}
			\hp^*\hp & = (\pt H_{\phi\eta})^* (\pt H_{\phi\eta})\\
			& = H_{\phi\eta}^* \pt H_{\phi\eta} \quad \text{(since $P_\theta^*=P_\theta$ and $P_\theta^2=P_\theta$)}\\
			& = H_{\phi\eta}^* H_{\breve{\theta}}H_{\bar{\theta}}H_{\phi\eta} \quad \text{(by}~\eqref{2eq})\\
			& = H_{\phi\eta}^* H_{\bar{\theta}}^*H_{\bar{\theta}}H_{\phi\eta}\\
			& = (H_{\bar{\theta}}H_{\phi\eta})^*(H_{\bar{\theta}}H_{\phi\eta}).
		\end{align*}
		Hence, we conclude that the study of the compactness of $\hp$ on $\eta H^2$ is equivalent to that of $H_{\bar{\theta}}H_{\phi\eta}$ on $H^2.$
	\end{proof}
	
	By Brown-Halmos (\cite{BH}, Theorem 8), the product of two Toeplitz operators $T_\phi$ and $ T_\psi$ is itself a Toeplitz operator if and only if either $\phi$ is co-analytic (that is, $\bar{\phi}\in H^\infty$) or $\psi$ is analytic (that is, $\psi\in H^\infty$), and in either case, $$T_\phi T_\psi= T_{\phi\psi}.$$ Consequently, it follows from equation \eqref{relation2} that $$H_{\breve{\phi}}H_\psi=0$$ if and only if either $H_{\breve{\phi}}=0$ or $H_\psi=0.$ This observation leads to the following theorem.
	
	\begin{theorem}\label{2t1}
		Let $\phi\in \li$ and $\eta,\theta$ be two inner functions with $\theta $ being non-constant. Then
		\begin{enumerate}[label=(\roman*)]
			\item the restricted Toeplitz operator $\tp$ is zero if and only if $\phi \in \bar{\eta}\theta H^\infty$,
			\vspace{0.1in}
			
			\item likewise, the restricted Hankel operator $\hp$ is zero if and only if $\phi \in \bar{\eta} H^\infty$.
		\end{enumerate}
	\end{theorem}
	
	\begin{proof}
		
		By Lemmas~\ref{2l1} and \ref{2l3}, the operators $\tp$ and $\hp$ are zero if and only if $$ H_{\breve{\theta}}H_{\phi\eta\bar{\theta}}=0 \quad \text{and} \quad H_{\bar{\theta}}H_{\phi\eta}=0,$$ respectively.
		Now, by the preceding discussion, $$ H_{\breve{\theta}}H_{\phi\eta\bar{\theta}}=0 \iff {\phi\eta\bar{\theta}}\in H^\infty \iff \phi \in \bar{\eta} \theta H^\infty.$$
		
		Equivalently, $$ H_{\bar{\theta}}H_{\phi\eta}=0 \iff H_{\phi\eta}=0 \iff  \phi \in \bar{\eta} H^\infty. $$ This completes the proof. 
	\end{proof}
	Note that, if $\theta$ is a uni-modular constant, then $\2=\{0\}$, therefore any operator having range in $\2$ is trivially zero. 
	\vspace{0.1in}
	
	To characterize finite-rank restricted Toeplitz and Hankel operators, we invoke the following classic results: the Kronecker theorem and the Axler-Chang-Sarason theorem.
	
	\begin{lemma}(Kronecker theorem \cite{MR, NKN, VVP}): \label{2l4}
		For $\phi \in \li,$ the Hankel operator $H_\phi$ (equivalently $\ch_\phi$) is of finite rank if and only if $$ \phi \in H^\infty + \mathcal{R},$$ where $\mathcal{R}$ is the set of rational functions, $\frac{p(z)}{q(z)}$ with zeros of $q$ lie inside $\mathbb{D}$.
	\end{lemma}
	
	\begin{lemma}(Axler-Chang-Sarason \cite{ACS}): \label{2l5}
		For $\phi,\psi \in \li$, the operator $ \ch_\phi^*\ch_\psi$  is of finite rank if and only if either $\ch_\phi$ or $\ch_\psi$ is of finite rank. 
		
		Equivalently, $ H_{{\phi}} H_\psi$ is of finite rank if and only if $H_{{\phi}}$ or $H_\psi$ is of finite rank.
	\end{lemma}
	Using the above two results, we obtain the following finite-rank characterization of RTOs and RHOs.
	
	\begin{theorem}\label{2t2}
		Let $\phi\in L^{\infty}$ and $\eta,\theta$ be two inner functions. Then
		
		\begin{enumerate}[label=(\roman*)]
			\item the restricted Toeplitz operator $\tp$ is of finite rank if and only if either $$\phi \in \bar{\eta}\theta (H^\infty+ \mathcal{R})$$ or the inner function $\theta$ is a finite Blaschke product.
			
			\item the restricted Hankel operator $\hp$ is of finite rank if and only if either $$ \phi \in\bar{\eta} (H^\infty+ \mathcal{R})$$ or the inner function $\theta$ is a finite Blaschke product.
		\end{enumerate}
	\end{theorem}
	\begin{proof}
		$(i)$ By Lemma~\ref{2l1}, the RTO $\tp$ is of finite rank if and only if $$H_{\breve{\theta}}H_{\phi\eta\bar{\theta}} $$ on $H^2$ is of finite rank. By Lemma~\ref{2l5}, this holds if and only if either $H_{\breve{\theta}}$ or $H_{\phi\eta\bar{\theta}}$ is of finite rank. Now by Lemma~\ref{2l4}, $H_{\breve{\theta}}$ is finite rank when $\theta$ is a finite Blaschke product, and conversely. Whereas, $H_{\phi\eta\bar{\theta}}$ is of finite rank if and only if $${\phi\eta\bar{\theta}}\in (H^\infty+ \mathcal{R}),$$ equivalently, $$\phi \in \bar{\eta}\theta (H^\infty+ \mathcal{R}).$$ This proves $(i)$.
		\vspace{0.1in}
		
		For $(ii)$, one can start with Lemma~\ref{2l3}, and proceed similarly as above to get the desired result for RHO $\hp$, so we omit the proof. 
	\end{proof}
	
	To characterize the compactness of RTOs and RHOs, we make extensive use of the theory of Douglas algebras. A \emph{Douglas algebra} is a closed subalgebra of $L^{\infty}$ that contains $H^\infty$. We denote the Gelfand space (the space of nonzero multiplicative linear functionals) of the Douglas algebra $A$ by $\mathcal{M}(A)$.
	
	If $m\in \mathcal{M}(H^{\infty})$, then $m$ can be viewed as a multiplicative linear functional on $H^{\infty}$. Therefore, by the Gleason-Whitney theorem \cite{D}, $ m$ admits a unique positive extension $l_m$ to a bounded linear functional on $L^{\infty}$. Thus, by the Riesz representation theorem, there exists a measure $d\mu_m$, called the representing measure, with support $S_m \subseteq \mathcal{M}(L^{\infty})$, such that 
	$$ l_m(f) = \int_{S_m} f d\mu_m. $$
	For more details, the reader is referred to \cite{KH}. A subset of $\mathcal{M}(L^{\infty})$ is called a support set if it is the (closed) support of the representing measure corresponding to a functional in $\mathcal{M}(H^{\infty}+C)$. For further information on $H^{\infty}$, $L^{\infty}$, and their maximal ideal spaces, see \cite{SC, JBG, DM, DES1}.
	\vspace{0.1in}
	
	A fundamental result due to Axler-Chang-Sarason \cite{ACS}, and later extended by Volberg \cite{V}, is stated below.
	
	\begin{lemma}\label{2l6}
		For the symbols $\phi,\psi \in L^\infty,$ $\ch_{\bar{\phi}}^*\ch_\psi$ is compact if and only if for each support set $S_m$, either $\bar{\phi}|_{S_m}$ or $\psi|_{s_m}$ is in $H^\infty |_{S_m}.$
		\vspace{0.1in}
		
		Equivalently, the product $H_\phi H_\psi$  of two Hankel operators is compact if and only if ${\phi}^*|_{S_m}$ or $\psi|_{s_m}$ is in $H^\infty |_{S_m}.$
	\end{lemma}
	We use this to obtain a characterization of the compactness of RTOs and RHOs.
	
	\begin{theorem}\label{2t3}
		Let $\phi\in L^{\infty}$. Then we have the following characterizations:
		\begin{enumerate}[label=(\roman*)]
			\item The restricted Toeplitz operator $\tp$ is compact if and only if $$\bar{\theta} |_{S_m}\in H^\infty |_{S_m} \quad\text{ or } \quad \phi \in \bar{\eta}\theta H^\infty |_{S_m}.$$
			\item The restricted Hankel operator $\hp$ is compact if and only if $$\bar{\theta}|_{S_m} \in H^\infty |_{S_m} \quad \text{ or } \quad \phi |_{S_m} \in \bar\eta H^\infty |_{S_m}.$$
		\end{enumerate}
	\end{theorem}
	
	\begin{proof}
		From Lemma \ref{2l1}, we get the RTO $\tp$ acting on $\eta H^2$ is compact if and only if $H_{\breve{\theta}}H_{\phi\eta\bar{\theta}}$ acting on $H^2$ is compact. Therefore by using Lemma~\ref{2l6}, we conclude that $H_{\breve{\theta}}H_{\phi\eta\bar{\theta}}$ is compact if and only if $\bar{\theta} |_{S_m}$ or $\phi \eta \bar{\theta}|_{S_m}$ is in $H^\infty |_{S_m}$, for each support set $S_m$.
		\vspace{0.1in}
		
		For the RHO $\hp$, using Lemmas \ref{2l3}, we have $\hp$ is compact if and only if $H_{\bar{\theta}}H_{\phi\eta}$ is compact. And, $H_{\bar{\theta}}H_{\phi\eta}$ is compact if and only if $\breve{\theta}|_{S_m} \in H^\infty |_{S_m}$ or $\phi\eta|_{S_m} \in H^\infty |_{S_m}.$ Hence the result follows.
	\end{proof}
	
	The following two corollaries are immediate consequences for the operators $\tau_\phi$ and $h_\phi$ due to equation \eqref{adj}. 
	
	\begin{corollary}\label{2c1}
		Let $\phi\in\li$ and $\eta, \theta$ are inner functons. Then,
		\begin{enumerate}[label=(\roman*)]
			\item $\tau_\phi=0$ if and only if $\phi \in \eta \overline{\theta H^\infty}$.
			\item $\tau_\phi$ is of finite rank if and only if $\phi \in\eta \overline{\theta(H^\infty+\mathcal{R}})$ or $\theta$ is a finite Blaschke product.
			\item $\tau_\phi$ is compact if and only if for each support set $S_m$, either $\bar{\theta}|_{S_m} \in H^\infty |_{S_m}$ or $\bar{\phi}|_{S_m} \in \bar\eta \theta H^\infty|_{S_m}$.
		\end{enumerate}
	\end{corollary}
	
	\begin{corollary}\label{2c2}
		Let $\phi\in\li$ and $\eta, \theta$ are inner functons. Then,
		\begin{enumerate}[label=(\roman*)]
			\item $h_\phi=0$ if and only if $ \phi \in \breve{\eta}{H^\infty}$.
			\item $h_\phi$ is of finite rank if and only if $\phi^* \in{\bar{\eta}(H^\infty+\mathcal{R})}$ or $\theta$ is a finite Blaschke product.
			\item $h_\phi$ is compact if and only if for each support set $S_m$, $\bar{\theta}|_{S_m} \in H^\infty |_{S_m}$ or $\phi |_{S_m} \in \breve\eta H^\infty |_{S_m}$.
		\end{enumerate}
	\end{corollary}

	\section{Characterizations}
	
	This section is devoted to characterizing restricted operators via operator equations. We begin by recalling the algebraic characterizations of classical Toeplitz operators and their truncated and dual variants.
	
	A bounded operator $X: H^2 \to H^2$ is a Toeplitz operator if and only if $S^* X S = X$, in which case $X = T_\phi$ for some $\phi \in L^\infty(\mathbb{T})$  (Brown--Halmos \cite{BH}). In contrast, the truncated Toeplitz on $\2$ behaves differently: a bounded operator $Y:\2\rightarrow\2$ is a truncated Toeplitz operator if and only if $Y-S_\theta Y S_\theta^*$ is a certain rank-two operator (Sarason \cite{DES}).
	
	In (\cite{DS}, 2018), Ding-Sang introduced the concept of dual truncated Toeplitz operator (DTTO) $D_\phi^\theta:\2 ^\perp\rightarrow\2^\perp$, defined by $D_\phi^\theta(h)=Q_\theta(\phi h),$ for  $\phi\in\li$. Later, C\^amara et al. in \cite{CKLP} as well as Gu in \cite{CG} provided some characterizations of DTTO in parallel to the Toeplitz and the truncated Toeplitz operators by showing that $D_\phi^\theta - U_\theta^* D_\phi^\theta U_\theta$ is a rank-two operator, where $U_\theta=Q_\theta M_z |_{ \2 ^\perp }$ is the dual compressed shift.
	
	We now turn to the characterization of restricted Toeplitz operators. We begin with the following lemma, whose proof is straightforward and hence omitted.
	
	\begin{lemma}\label{3l1}
		If $k$ is an element of $\2$, then $S^*(k)= S_\theta^*(k)$. 
	\end{lemma}
		\vspace{0.1in}
		
		$\bullet$ {\bf{{\emph{Proof of the Theorem \ref{3th1}:}}}}
		\vspace{0.1in}
		
		From Lemma \ref{2l1}, it follows that the action of $\tp$ on $\eta H^2$ is equivalent to the action of $H_{\breve{\theta}} H_{\bar{\theta} \phi\eta}$ on $H^2.$ Therefore, we have the following.
		\begin{align*}
			\tp  = H_{\breve{\theta}} H_{\bar{\theta} \phi\eta}  \implies S_\theta^* \tp \mathcal{S}_{(\eta)} & = S^* \tp \mathcal{S}_{(\eta)} \quad \text{ (by Lemma~\ref{3l1}) }\\
			&= S^* H_{\breve{\theta}} H_{\bar{\theta} \phi\eta} S\\
			& = H_{\breve{\theta}} S S^* H_{\bar{\theta} \phi\eta}\\
			& = H_{\breve{\theta}} (I-e_0\otimes e_0) H_{\bar{\theta} \phi\eta} \quad \text{ (here $e_n=z^n, n\geq0$) }\\
			& = \tp - H_{\breve{\theta}}(e_0) \otimes H_{\bar{\theta} \phi\eta }^*(e_0)\\
			& = \tp - P(\bar{z} \theta) \otimes P(\bar{z} \theta \overline{\phi\eta}).
		\end{align*}
		Thus we have \begin{equation}\label{3eq1}
			\tp - S_\theta^* \tp \mathcal{S}_{(\eta)} = P(\bar{z} \theta) \otimes P(\bar{z} \theta \overline{\phi\eta}) = S^* (\theta) \otimes P(\theta \overline{z\phi\eta}).\end{equation}
		For the converse part, let us assume that the operator $X:\eta H^2\rightarrow\2$ satisfies 
		\begin{equation}\label{3eq2}
			X- S_\theta^* X \mathcal{S}_{(\eta)} = S^* (\theta) \otimes P(\theta \overline{z\phi\eta}),
		\end{equation}
		for some $\phi\in\li$. 
		
		Now by subtracting equation \eqref{3eq1} from equation \eqref{3eq2}, we get 
		\begin{equation}\label{3eq3}
			S_\theta^* (X-\tp) \mathcal{S}_{(\eta)} = (X-\tp) .
		\end{equation}
		Therefore, the solution of the equation $ S_\theta^* A \mathcal{S}_{(\eta)} = A $ for some $A:\eta H^2\rightarrow\2$ gives the solution of the equation \eqref{3eq3}. 
		
		The operator $A:\eta H^2\rightarrow\2$ can be viewed as $A:H^2\rightarrow H^2,$ with $A|_{K_\eta}=0$ and $range(A) \subset \2,$ which turns the equation $ S_\theta^* A \mathcal{S}_{(\eta)} = A $ into $ S^* A S = A $ in $H^2$. So, the Brown-Halmos identity for the Toeplitz operator shows that $A$ is nothing but the Toeplitz operator $T_\psi,$ for some $\psi\in \li,$ with $T_\psi|_{K_\eta}=0$ and $range(T_\psi)\subset \2.$ 
		
		But the only Toeplitz operator $T_\psi$ with $range(T_\psi)\subset \2$ is zero. Indeed, consider $T_\psi$ with $range(T_\psi)\subset \2$. Then we have the following.
		\begin{align*}
			range(T_\psi)\subset \2 & \implies \langle T_\psi h, \theta f\rangle = 0, \quad \text{ for all $h, f\in H^2$ }\\
			& \implies \langle \bar{\theta}\psi h, f \rangle = 0, \quad \text{ for all $h, f\in H^2$ }\\
			& \implies \bar{\theta}\psi h \perp H^2, \quad \text{ for all $h\in H^2$ }\\
			& \implies P(\bar{\theta}\psi h)=0 \quad \text{ for all $h\in H^2$. }
		\end{align*}
		Set $h=z^n$ for $n\geq0,$ then we have $P(\bar{\theta}\psi z^n)=0$ for $n\geq0.$ This shows that $\bar{\theta}\psi$ has no negative Fourier coefficient, and thus $\bar{\theta}\psi\in H^2.$ So, at the end, we have $\psi$ is an $H^{\infty}$ element. 
		
		Therefore, $T_{\psi}(H^2)=\psi H^2\subset \2$, which is possible only when $\psi$ is none other than zero, and hence $T_\psi=0.$ 
		
		Also, note that one can also establish the above result by applying a more general theorem provided by D. Vukotić in \cite{DV}.
		
		Hence, the only solution to the operator equation $S_\theta^* A \mathcal{S}_{(\eta)} = A$ derived from equation \eqref{3eq3} is zero, therefore $X=\tp$ is the solution to the equation \eqref{3eq3}. This completes the proof. \hspace{4in} $\square$
		
		As a consequence, we get a characterization of the operator $\tau_\phi$ as follows.
		\begin{corollary}
			Let $\eta, \theta$ be two inner functions, $\phi \in \li $, and let $Y:\2 \rightarrow \eta H^2$ be a bounded linear operator. Then, $Y= \tau_\phi$ if and only if 
			$$ Y - \mathcal{S}_{(\eta)}^* Y S_\theta = P(\theta\phi \overline{z\eta}) \otimes S^*(\theta). $$
		\end{corollary}
		
		\begin{proof}
			The result follows by taking the adjoint of both sides of the equation \eqref{3eq1} and replacing $\phi$ with $\bar{\phi}$ (since $\tp^*= \tau_{\bar{\phi}}$).
		\end{proof}
		
		Now moving forward, we call an RTO analytic if its corresponding symbol $\phi$ is analytic, that is, $\phi\in H^\infty$. A classical result of Brown--Halmos (\cite{BH}, Theorem~7) asserts that a bounded operator $A: H^2\rightarrow H^2$ is an analytic Toeplitz operator if and only if $AS=SA$. 
		
		Subsequently, Sarason (\cite{SAR}, Theorem 1) showed, using the commutant lifting theorem, that every bounded operator on $\2$ commuting with $S_\theta$ is an analytic truncated Toeplitz operator, and conversely. In addition, an alternative proof of this result was later given by Nikolski \cite{NKN} (see Theorem~3.1.11, Part~B).
		
		Motivated by these results, we now provide a characterization of a certain kind of restricted Toeplitz operator, stated in Theorem \ref{3thm}.
		\vspace{0.1in}
		
		$\bullet$ {\bf{{\emph{Proof of the Theorem \ref{3thm}:}}}}
		\vspace{0.1in}
		
		Firstly, if $A= \tp,$ with $\phi= \bar{\eta}\psi$, for $\psi\in H^\infty$, then for any $f\in  H^2,$ we have: 
		\begin{align*}
			S_\theta A(\eta f) &= \pt (z \tp(\eta f))\\
			& = \pt (z\pt(\phi \eta f))\\ 
			& = \pt ( z (I-M_\theta P M_{\bar{\theta}})(\psi f) ) \quad \text{ (recall: $\pt|_{H^2} = I - M_\theta P M_{\bar{\theta}} $) } \\
			& = \pt (z\psi f) - \pt (z \theta P({\bar{\theta}}\psi f) )\\
			& = \pt (\phi \mathcal{S}_{(\eta)}(\eta f))-0 \quad \text{ (since, $z\theta H^2 \subset \theta H^2$ )}~ \\&= \tp S (f).
		\end{align*}
		This shows that if $A= \tp,$ with $\phi\in \bar{\eta}H^\infty$, then $S_\theta A = A\mathcal{S}_{(\eta)}$. 
		
		For the converse, we provide two separate proofs; the first is motivated by \cite{NKN}, and the second one uses the Intertwining Lifting Theorem (see Theorem 2.3 of the monograph \cite{Book}). 
		\vspace{0.1in}

		{\textbf{\emph{First Proof}:}} Suppose $A: \eta H^2\rightarrow\2$ is a bounded operator with $S_\theta A = A\mathcal{S}_{(\eta)}$. Consider $A(\eta e_0)= A(\eta) = \alpha$, so $\alpha\in \2.$ 
		
		Now, for any $n\in \mathbb{N}$, we have 
		$$ A(\eta z^n)=  A\mathcal{S}_{(\eta)}^n(\eta) = S_\theta ^n A(\eta) = S_\theta^n (\alpha) = \pt(z^n \alpha).$$ 
		
		Therefore, it is not difficult to show that 
		\begin{equation}\label{3eq}
			A(\eta \theta) = \pt(\eta\theta \alpha) = 0.
		\end{equation}
		
		
		Now, construct an operator
		\begin{align*}
			& A_* :  H^2\rightarrow\ho \\
			& f \longrightarrow M_{\bar{\theta}} A M_\eta f.
		\end{align*} 
		This indeed is a well-defined map, since ${\bar{\theta}} \2 = \overline{z \2}\subset \ho$. Now we show that $A_*$ is a Hankel operator.
		
		Starting with $S_\theta A = A\mathcal{S}_{(\eta)},$ we have
		\begin{align*}
			S_\theta A = A\mathcal{S}_{(\eta)}  & \iff \pt S A M_\eta = A \mathcal{S}_{(\eta)}M_\eta ~ \text{ (on $H^2$) }  \\
			& \iff  (I - M_\theta P M_{\bar{\theta}}) S A M_\eta = A M_\eta S \quad \text{ (as $\mathcal{S}_{(\eta)}M_\eta= M_\eta S$) }\\
			& \iff M_\theta (I-P) M_{\bar{\theta}} S AM_\eta = A M_\eta S\\
			& \iff (I-P) M_z M_{\bar{\theta}} A M_\eta = M_{\bar{\theta}} A M_\eta S\\
			&\iff (I-P) M_z A_* = A_*S.
		\end{align*}
		Thus, $A_*$ is a Hankel operator. Therefore, by the Nehari's theorem \cite{ZN}, there is $\xi\in \li$ such that $A_*= \ch_{\xi}$.
		
		Now, consider $\theta \xi = \phi + \psi$, where $\phi \in H^2$ and $\psi \in \ho.$ Then we have 
		\begin{align*}
			& \psi = (I-P)(\theta \xi)\\
			& = \ch_\xi(\theta) = A_*(\theta) \\
			& = M_{\bar{\theta}} AM_\eta(\theta) = M_{\bar{\theta}} A(\eta\theta) = 0 \quad \text{ (by equation \eqref{3eq}). }
		\end{align*}
		
		Therefore, we conclude that $\theta\xi = \phi \in H^2 \cap \li= H^\infty$. 
		
		Rewrite the identity as $\overline{\theta} \phi = \xi$, and for any $f\in H^2$, we have the following:
		\begin{align*}
			A(\eta f) & = M_\theta M_{\bar{\theta}} A M_\eta (f)\\
			& = M_\theta A_* (f) \\
			& = M_\theta \ch_{\xi} (f) \\
			& = M_\theta \ch_{\phi \overline{\theta}}(f) \\
			& = M_\theta (I-P) M_{\phi \overline{\theta}}(f)\\
			& = M_\theta (I-P) M_{\bar{\theta}} (\phi f)\\
			& = \pt (\phi \bar{\eta} \eta f) = \mathcal{T}_{\bar{\eta}\phi}(\eta f).
		\end{align*}
		This completes the (first) proof.
		\vspace{0.1in}
		
		{\textbf{\emph{Second Proof}:}} We first recall the intertwining lifting theorem.
		\vspace{0.1in}
		
		(\emph{Intertwining Lifting Theorem}:) \emph{Let $\h_1, \h_2, \mathcal{K}_1 \text{ and }\mathcal{K}_2$ be four Hilbert spaces. Let $A_1 \in \mathcal{B}(\h_1)$ and let $A_2 \in \mathcal{B}(\h_2)$ be two contractions. Let $V_1 \in \mathcal{B}(\mathcal{K}_1)$ and  $V_2 \in \mathcal{B}(\mathcal{K}_2)$ be minimal isometric dilations of $A_1 \text{ and } A_2$, respectively. Then for an operator $X \in \mathcal{B}(\h_2,\h_1)$ satisfying $XA_2=A_1X$, there exists $Y \in \mathcal{B}(\mathcal{K}_2,\mathcal{K}_1)$ such that $YV_2=V_1Y,~||Y||=||X||,~X=P_{\h_1} Y|_{\h_2}$ and $Y(\mathcal{K}_2 \ominus \h_2) \subseteq \mathcal{K}_1 \ominus \h_1.$}
		\vspace{0.1in}

		Let $A: \eta H^2\rightarrow\2$ be a bounded operator such that $ A\mathcal{S}_{(\eta)}=S_\theta A$.
		Note that $S_\theta$ is a model operator, and hence, a minimal isometric dilation of $S_\theta$ is $S$ on $H^2.$ For, $S^*_\theta=S^*|_{\2}$ implies $S_\theta^n=P_{\2}S^n|_{\2}$ for $n\in \mathbb{N}\cup\{0\}.$ Let $\mathcal{H}_0$ be the linear span of $\{S^n\2: ~n\in \mathbb{N} \cup \{0\}\}.$ Let $h\in \mathcal{H}_0,$ then
		$$ h= \sum_{n=0}^{N}S^nk_n
		\implies Sh= \sum_{n=0}^{N}S^{n+1}k_n,
		$$ \text{ and }
		$$ S^*h= S_\theta ^*k_0+\sum_{n=1}^{N}S^{n-1}k_n,~ k_n \in \2.$$
		The above shows that $\bigvee_{n=0}^{\infty}S^n\2$ is a $S$-reducing subspace of $H^2,$ and hence $\bigvee_{n=0}^{\infty}S^n\2=H^2$.
		Since $\mathcal{S}_{(\eta)}$ is an isometry, it is a minimal isometric dilation of $\mathcal{S}_{(\eta)}$ itself. Therefore, by the above Intertwining Lifting Theorem, there exists a bounded linear operator $B: \eta H^2 \to H^2$ such that $$B\mathcal{S}_{(\eta)}=SB \quad \text{ and } \quad A=P_\theta B.$$
		
		\[
		\begin{tikzcd}[row sep=1em]
			{\eta H^2} \arrow[rr,"\mathcal{S}_{(\eta)}"] \arrow[dr,swap,"B"] \arrow[dd,swap,"I_{\eta H^2}"] &&
			{\eta H^2} \arrow[dd,swap,"I_{\eta H^2}" near start] \arrow[dr,"B"] \\
			& H^2 \arrow[rr,crossing over,"S" near start] &&
			H^2 \arrow[dd,"P_{\theta}"] \\
			{\eta H^2} \arrow[rr,"\mathcal{S}_{(\eta)}" near end] \arrow[dr,swap,"A"] &&{\eta H^2} \arrow[dr,swap,"A"] \\
			& K_\theta \arrow[rr,"S_{\theta}"] \arrow[uu,<-,crossing over,"P_{\theta}" near end]&& K_\theta
		\end{tikzcd}
		\]
		
		That is, for $g \in H^2,$
		\begin{align*}
			B\mathcal{S}_{(\eta)}(\eta g) = SB(\eta g)& \implies BS(\eta g)=SB(\eta g)\\
			& \implies BST_\eta (g)=SBT_\eta(g)\\
			& \implies BT_\eta S(g)= SBT_\eta(g).                
		\end{align*}
		This gives $$BT_\eta=T_\phi \quad \text{ for some $\phi \in H^\infty.$ }$$ 
		So we can write 
		$$B=T_{\phi \overline{\eta}}|_{\eta H^2} \quad \text{ for some $\phi \in H^\infty.$ }$$ 
		Moreover,
		$$A=P_\theta B \implies A=\mathcal{T}_{\phi{\overline{\eta}}} \quad \text{for some {$ \phi \in H^\infty$},}$$ equivalently, $A=\tp$ for $\phi \in \overline{\eta} H^\infty.$ This completes the (second) proof. \hspace{0.5in} $\square$
		\vspace{0.1in}
		
		In particular, if we consider the inner function $\eta$ to be a unimodular constant, we have the following remark.
		\begin{remark}
			Let $A: H^2\rightarrow\2$ be a bounded linear operator. Then $A$ is an analytic restricted Toeplitz operator $\tp$ if and only if $ AS=S_\theta A$. 
		\end{remark}
		
		We now turn to the characterization of restricted Hankel operators. Recall that a bounded operator $X: H^2 \to H^2$ is a Hankel operator if and only if $S^*X = XS$, in which case $X = H_\phi$ for some $\phi \in L^\infty(\mathbb{T})$ (Nehari’s theorem \cite{ZN}).
		
		In the truncated setting, Gu-Ma in \cite{GM} showed that $ S_\theta^*B_\phi^\theta-B_\phi^\theta S_\theta$ is a certain rank two operator, where $B_\phi^\theta$ is a truncated Hankel operator. In addition, sufficient conditions for an operator to be a truncated Hankel operator are also discussed.
		
		More recently, in 2025, the second and third named authors of this article initiated the study of the dual truncated Hankel operators (DTHO) $\h_\phi^\theta$ in \cite{CJ}. It is defined by $$\h_\phi^\theta(h)= Q_\theta J(\phi h) \quad \text{for $\phi\in \li$},$$
		and $J: L^2\rightarrow L^2$ is another flip operator $Jf(z)=f(\bar z).$ In connection with the Hankel and the truncated Hankel operators, it is shown in \cite{CJ} that $U_\theta^* \h_\phi^\theta - \h_\phi^\theta U_\theta$ is a rank two operator. Conversely, the sufficient conditions for an operator to be DTHO have also been derived. 
		
		We now present the following set of results that lead to a characterization of the restricted Hankel operators. We begin with the following lemma.
		
		\begin{lemma}\label{3l2}
			Suppose $\phi\in\li$ and let $\hp:\eta H^2 \rightarrow\2$ be the restricted Hankel operator. Then $S_\theta^* \hp - \hp \mathcal{S}_{(\eta)}$ is at most a rank-one operator. In fact, 
			\begin{equation}\label{3eq4}
				\hp \mathcal{S}_{(\eta)} -  S_\theta^*\hp = S^*(\theta) \otimes \Phi, \quad \text{ where $\Phi=P(\breve{\theta} ~\overline{z\phi\eta})$.}
			\end{equation}
		\end{lemma}
		
		\begin{proof}
			First, observe that the action of $\hp $ on $\eta H^2$ is same as the action of $\h_{\phi\eta}$ on $H^2,$ and recall $\h_\psi= \pt H_\psi = (I-M_\theta P M_{\bar{\theta}})H_\psi$. So we have the following. 
			\begin{align*}
				&\h_{\phi\eta}\mathcal{S}_{(\eta)} - S_\theta^* \h_{\phi\eta}\\
				& = \h_{\phi\eta}S - S^* \h_{\phi\eta} \quad \text{ (by Lemma~\ref{3l1})}\\
				&=(H_{{\phi\eta}} - T_\theta H_{{\phi\eta}\theta^*}) S - S^* (H_{{\phi\eta}} - T_\theta H_{{\phi\eta}\theta^*})\\
				&=S^*T_\theta H_{{\phi\eta}\theta^*} -  T_\theta H_{{\phi\eta}\theta^*}S\\ 
				&=S^*T_\theta (SS^* + (e_0\otimes e_0)) H_{{\phi\eta}\theta^*} -  T_\theta H_{{\phi\eta}\theta^*}S\\
				&=S^*T_\theta S S^*  H_{{\phi\eta}\theta^*} + S^*T_\theta (e_0\otimes e_0) H_{{\phi\eta}\theta^*} -  T_\theta H_{{\phi\eta}\theta^*}S\\
				&=S T_\theta H_{{\phi\eta}\theta^*} +  S^*T_\theta (e_0\otimes e_0) H_{{\phi\eta}\theta^*} - T_\theta H_{{\phi\eta}\theta^*}S\\
				&=S^*T_\theta (e_0) \otimes H_{{\phi\eta}\theta^*}^*(e_0)\\
				&=S^*(\theta) \otimes P\f (\theta \phi^*\eta^*) = S^*(\theta) \otimes P(\breve{\theta}~\overline{z\phi\eta}).
			\end{align*}
			Thus, $ \hp \mathcal{S}_{(\eta)} -  S_\theta^*\hp = S^*(\theta) \otimes \Phi.$
		\end{proof}
		
		As a consequence of the above lemma, we get $ \hp \mathcal{S}_{(\eta)} -  S_\theta^*\hp = 0$ if and only if $\phi \in \breve{\theta}\bar{\eta}H^2$. Moreover, Theorem~\ref{2t1} assures a large class of non-zero RHOs satisfying $ \hp \mathcal{S}_{(\eta)} = S_\theta^*\hp $.
		\vspace{0.1in}
		
		The following lemma will be useful in establishing the converse part of the characterization of RHO.
		
		\begin{lemma}\label{3l3}
			Let $\phi\in \li$ and let $\eta, \theta$ be two inner functions. Suppose $H_{\phi\eta}$ is a Hankel operator on the Hardy-Hilbert space $H^2.$ Then the range of $\h_{\phi\eta}$ is contained in $\2$ if and only if $\phi \in \breve{\theta}\bar{\eta}H^\infty.$
		\end{lemma}
		\begin{proof}
			By the hypothesis, the range of $H_{\phi\eta}$ is contained in $\2$ if and only if $range(H_{\phi\eta}) \perp \theta H^2.$ Therefore, 
			\begin{align*}
				& \langle H_{\phi\eta}(h), \theta f \rangle=0 \text{ for all $f,h \in H^2$}\\
				&\iff \langle  \f (\phi\eta h) , {\theta} f \rangle =0 \\
				&\iff \langle \bar{\theta} \f (\phi\eta h) , f\rangle=0 \\
				&\iff \langle \f( \theta^* \phi\eta h) , f\rangle = 0 \text{ for all $f,h \in H^2$} \\
				& \iff \f( \theta^* \phi\eta h) \in \ho \text{ for all $h \in H^2$}\\
				& \iff \theta^* \phi\eta h \in H^2 \text{ for all $h \in H^2$}\\
				& \iff \theta^* \phi\eta H^2 \in H^2 \\
				& \iff \theta^* \phi\eta \in H^\infty \quad \text{($\psi H^2 \subset H^2 \iff \psi \in H^\infty$)}\\
				& \iff \phi \in \breve{\theta}\bar{\eta}H^\infty.
			\end{align*} This completes the proof.
		\end{proof}
		
		Recall that the Hankel operator $H_\phi$ acting on $\eta H^2$ can be identified with the Hankel operator $H_{\phi\eta}$ acting on $H^2$. Moreover, when $range(H_{\phi\eta}) \subset \2$, then 
		$$H_{\phi}(h)=P\f(\phi h)= \pt \f (\phi h)= \hp(h).$$
		In this case, it also evident that $S^* H_\phi = H_\phi S = \hp \mathcal{S}_{(\eta)} = S_\theta^* \hp$ (follows from the Lemma \ref{3l3}). We are now ready to prove Theorem \ref{3th2}.
		\vspace{0.1in}
		
		$\bullet$ {\bf{\emph{Proof of the Theorem \ref{3th2}:}}}
		\vspace{0.1in}
		
		The proof follows a similar line of argument as in Theorem \ref{3th1}. The first part has already been established in the Lemma \ref{3l2}, that is, the restricted Hankel operator $\hp$ satisfies $$\hp \mathcal{S}_{(\eta)} - S_\theta^*\hp = S^*({\theta}) \otimes P(\breve{\theta}~{\overline{z\eta\phi}}).$$
		For the converse part, let $X:\eta H^2\rightarrow\2$ be such that $$X\mathcal{S}_{(\eta)} - S_\theta^*X = S^*({\theta}) \otimes P(\breve{\theta}~{\overline{z\eta\phi}}),$$ for some $\phi\in \li.$
		Thus, we get,
		\begin{align*}
			& X\mathcal{S}_{(\eta)} - S_\theta^*X = \hp \mathcal{S}_{(\eta)} - S_\theta^*\hp\\
			& \iff (X-\hp)\mathcal{S}_{(\eta)} - S_\theta^*(X-\hp) = 0\\
			& \iff X-\hp = \h_\psi \quad \text{ for some $\psi \in \breve{\theta}\bar{\eta}H^\infty$ \quad (due to lemma \ref{3l3})}\\
			& \iff X = \h_{\phi+\psi}.
		\end{align*} 
		Therefore, $X$ is a restricted Hankel operator, hence proved. \hspace{1in} $\square$

		\begin{corollary}
			Let $\eta, \theta$ be two inner functions, and let $Y:\2 \rightarrow \eta H^2$ be a bounded linear operator. Then, $Y=h_\phi$ if and only if 
			$$ \mathcal{S}_{(\eta)}^*Y - Y S_\theta = P( \breve{\theta} \breve{\phi} \overline{z \eta}) \otimes S^*(\theta), \quad \phi\in\li. $$
		\end{corollary}
		\begin{proof}
			The result follows from the equation \eqref{3eq4} and the relation $\hp^* = h_{\phi^*}$.
		\end{proof}

		\section{Small and Big Truncated Toeplitz operators}
		
		This section contains the study of zero, finite rank, and compactness of small and big truncated Toeplitz operators, defined in \ref{Def}. In this study, we make use of the previously defined operators $\tau_\phi$, $h_\phi$ (see \eqref{1} and \eqref{2}) together with the results established earlier in Section 2. We begin with the study of STTO by decomposing it as a sum of two operators as follows.
		\begin{lemma}\label{4l1}
			For $\phi \in L^\infty$, the small truncated Toeplitz operator $t_\phi^\theta$ admits the representation $$ t_\phi^\theta = \tau_\phi +  M_{\bar{\theta}} \f h_{\theta \phi}.$$
		\end{lemma}
		\begin{proof}
			By definition, the small truncated Toeplitz operator $t_\phi^\theta: \2 \rightarrow (\overline{z\2})^\perp$ is defined by $t_\phi^\theta = (I-P_{\bar{\theta}})M_\phi|_{\2},$ for $\phi\in\li.$ Since the orthogonal projection $P_{\bar{\theta}}: L^2 \rightarrow \overline{z\2}$ can be expressed as $$P_{\bar{\theta}} = Q-M_{\bar{\theta}} Q M_\theta \quad \text{ (where $Q=I-P$) },$$ we have the following
			\begin{align*}
				t_\phi^\theta & = (I-P_{\bar{\theta}})M_\phi|_{\2}\\
				& = \Bigl(I-\left( Q- M_{\bar{\theta}} Q M_\theta \right) \Bigr) M_\phi|_{\2}\\
				& = (I-Q)M_\phi|_{\2} +  M_{\bar{\theta}} Q M_\theta M_\phi|_{\2}\\
			\end{align*}
			Using the identity $Q=\f P \f$, we obtain
			\begin{align*}
				& t_\phi^\theta = \tau_\phi + M_{\bar{\theta}} \f P\f M_{\theta \phi}|_{\2}\\
				&= \tau_\phi +  M_{\bar{\theta}} \f h_{\theta \phi}.
			\end{align*}
		\end{proof}
		
			
			We now state the following well-known result, which will be useful for the upcoming theorems.
			
			\begin{lemma}\label{4}
				Let $\h$ be a Hilbert space, $U$ be an isometry, and $A:\h\rightarrow\h$ be a bounded operator. Then $UA$ is zero, finite rank, and compact if and only if $A$ is zero, finite rank, and compact, respectively.
			\end{lemma}
			
			Since $M_{\bar{\theta}}\f$ is an isometry, it follows from the above lemma that STTO $t_\phi^\theta$ is zero, finite rank, and compact if and only if both $\tau_\phi$ and $h_{\theta \phi}$ are zero, finite rank, and compact together.

			\begin{theorem}\label{4th1}
				The small truncated Toeplitz operator $t_\phi^\theta$ 
				\begin{enumerate}[label=(\roman*)]
					\item is zero if and only if $\phi$ is a constant multiple of $\bar{\theta},$ 
					\vspace{0.1in}
					
					\item is finite in rank if and only if either $\theta$ is a finite Blaschke product or $\phi = \bar{\theta}(q_1+\bar{q_2}),$ where $q_1,q_2$ are the rational functions having poles outside the closure of $\mathbb{D}$,\label{41}
					\vspace{0.1in}
					
					\item is compact if and only if for each support set $S_m$ either $\bar{\theta}|_{S_m}$ is constant or $\phi|_{S_m} = \bar{\theta}|_{S_m}.$
				\end{enumerate}
			\end{theorem}
			\begin{proof}
				The foregoing discussion shows that the STTO $t_\phi^\theta=0$ if and only if both $\tau_\phi =0 $ and $h_{\theta \phi} = 0$. And Corollaries \ref{2c1} and \ref{2c2} yield that $\tau_\phi =0 $ if and only if $ \phi \in \overline{\theta H^\infty}$, whereas $h_{\theta \phi} = 0$ if and only if $\phi \in \bar{\theta}H^\infty$. Thus, the result follows. 
				\vspace{0.1in}
				
				On a similar note, $t_\phi^\theta$ is finite in rank if and only if $\tau_\phi$ and $h_{\theta \phi}$ both are of finite rank. Due to Corollary \ref{2c1}, $\tau_\phi$ is of finite rank if and only if either $\theta$ is finite Blaschke product or $\phi \in \overline{\theta(H^\infty+\mathcal{R}}) \iff \phi\theta \in \overline{(H^\infty+\mathcal{R}})$. And Corollary \ref{2c2} gives that $h_{\theta \phi}$ is of finite rank if and only if either $\theta$ is a finite Blaschke product or $\phi \in (H^\infty+\mathcal{R})$. Therefore, STTO $t_\phi^\theta$ is finite in rank if and only if either $\theta$ is a finite Blaschke product or $\phi = \bar{\theta}(q_1+\bar{q_2}),$ where $q_1,q_2$ are as described in \ref{41}.
				\vspace{0.1in}
				
				The proof of compactness follows analogously. The compactness of $\tau_\phi$ follows from Corollary~\ref{2c1} (iii), that is, for each support set $S_m$, either $\bar{\theta}|_{S_m} \in H^\infty |_{S_m}$ or $\bar{\phi}|_{S_m} \in  \theta H^\infty|_{S_m}$, which implies $\overline{\theta \phi}|_{S_m} \in  H^\infty|_{S_m}$. And Corollary \ref{2c2} show that $h_{\theta \phi}$ is compact if and only if either $\bar{\theta}|_{S_m} \in H^\infty |_{S_m}$ or $\theta \phi |_{S_m} \in  H^\infty |_{S_m}$. Since $\theta$ is an inner function, we have $\theta|_{S_m} \in H^\infty|_{S_m}$. Therefore, combining $\theta|_{S_m} \in H^\infty|_{S_m}$ and $\bar{\theta}|_{S_m} \in H^\infty |_{S_m}$ gives $\bar{\theta}|_{S_m}$ is constant, whereas $\overline{\theta \phi}|_{S_m} \in  H^\infty|_{S_m}$ and $\theta \phi |_{S_m} \in  H^\infty |_{S_m}$ yield $\theta \phi |_{S_m}$ is constant. Hence, the result follows.
			\end{proof}
			
			Note that in Lemma \ref{4l1}, we presented STTO as a combination of $\tau_\phi$ and $h_{\theta \phi}$. On the other hand, the same conclusions about $t_\phi^\theta$ can also be drawn by following the approach of Ma--Yan--Zheng \cite{MYZ}. To this end, we decompose the STTO in terms of classical Hankel operators and their adjoints as follows:
			\begin{align*}
				t_\phi^\theta = & (I-Q)M_\phi|_{\2} +  M_{\bar{\theta}} Q M_\theta M_\phi|_{\2}\\
				& = PM_\phi|_{\2} +  M_{\bar{\theta}}\ch_{\phi\theta}|_{\2}\\
				& = \ch_{\overline{\phi\theta}}^* \ch_{\bar{\theta}} +  M_{\bar{\theta}}\ch_{\phi\theta}\ch_{\bar{\theta}}^* \ch_{\bar{\theta}}.
			\end{align*}
			Now, by applying Lemma~2.2 of \cite{MYZ}, it follows that the vanishing, finite-rank, and compactness properties of STTO $t_\phi^\theta$ are determined by the corresponding properties of the operators  $$\ch_{\overline{\phi\theta}}^* \ch_{\bar{\theta}} \quad \text{and} \quad\ch_{\bar{\theta}}^*\ch_{\phi\theta},$$
			respectively. Hence, the results of Theorem \ref{4th1} can also be recovered from this.
			
			On a similar note, the big truncated Hankel operator $H_\phi^\theta$ can be expressed as $$ H_\phi^\theta = \f h_{\phi} + M_\theta \tau_{\bar{\theta}\phi} .$$ 
			Consequently, by Lemma~\ref{4} together with Corollaries~\ref{2c1} and \ref{2c2}, the Theorems~1.1, 1.2, and 1.3 of \cite{MYZ} can be restored.
			\vspace{0.1in}
			
			Next, we conduct the same study of compactness of all kinds for the big truncated Toeplitz operators. We begin by breaking down the BTTO into its constituent parts.
			
			\begin{lemma}\label{4l2}
				The big truncated Toeplitz operator $T_\phi^\theta: \2 \rightarrow(\2)^\perp,$ defined by $T_\phi^\theta(f) = (I-\pt)\f (\phi f)$, for $\phi\in \li,$ admits the representation $\f \tau_{\phi} + M_\theta h_{\phi\theta^*}$.
			\end{lemma}
			\begin{proof}
				By definition, we have $T_\phi^\theta(f) = (I-\pt)\f (\phi f)$. From this, we derive the following decomposition
				\begin{align*}
					T_\phi^\theta |_{\2} &= (I-\pt)\f M_\phi|_{\2}\\
					& = \Bigl(I - ( P - M_\theta PM_{\bar{\theta}} )\Bigr)\f M_\phi|_{\2} \\
					& = (I-P) \f M_\phi|_{\2} + M_\theta PM_{\bar{\theta}}\f M_\phi|_{\2} \\
					& = \f P M_\phi|_{\2} + M_\theta P\f M_{\phi{\theta}^*}|_{\2} \\
					& = \f \tau_{\phi} + M_\theta h_{\phi\theta^*}.
				\end{align*}
			\end{proof}

			Since $\f$ is a unitary operator and $M_\theta$ is an isometry, we conclude that $T_\phi^\theta$ acting on $\2$ is zero, finite rank, and compact if and only if both $\tau_{\phi}$ and $h_{\phi\theta^*}$ are simultaneously zero, finite rank, and compact, due to Lemma \ref{4}. Thus, we arrive at the following results.
			\begin{theorem}
				The big truncated Toeplitz operator $T_\phi^\theta$ is 
				\begin{enumerate}[label=(\roman*)]
					\item zero if and only if $\phi\in\overline{\theta H^\infty}\cap \breve{\theta}H^\infty,$
					\item of finite in rank if and only if $\phi \in \overline{\theta(H^\infty+\mathcal{R})}\cap \breve{\theta}(H^\infty + \mathcal{R}).$
				\end{enumerate}
			\end{theorem}
			
			\begin{proof}
				Based on the preceding discussion, we have that the BTTO $T_\phi^\theta=0$ if and only if $\tau_{\phi}=0$ and $h_{\phi\theta^*}=0.$ Now, the Corollaries \ref{2c1} and \ref{2c2} yield that $\tau_\phi=0 \iff \phi \in \overline{\theta H^\infty}$. On the other hand, $h_{\phi\theta^*}=0 \iff {\phi\theta^*} \in H^\infty\iff \phi\in \breve{\theta}H^\infty.$ Hence, the result follows.
				\vspace{0.1in}
				
				Similarly, it follows from Corollary \ref{2c1} that the operator $\tau_\phi$ is of finite rank if and only if $\phi\in \overline{\theta(H^\infty+\mathcal{R})}$. Furthermore, by Corollary \ref{2c2}, we also have $h_{\phi\theta^*}$ is of finite rank if and only if $\phi\theta^* \in H^\infty + \mathcal{R}\iff \phi \in \breve{\theta}(H^\infty + \mathcal{R})$. Hence, the BTTO is finite in rank if and only if $\phi \in \overline{\theta(H^\infty+\mathcal{R})}\cap \breve{\theta}(H^\infty + \mathcal{R}).$
			\end{proof}
			
			The following theorem characterizes the compactness of the BTTO.
			\begin{theorem}
				The big truncated Toeplitz operator $T_\phi^\theta$ is compact if and only if one of the following conditions holds: 
				\begin{enumerate}[label=(\roman*)]
					\item $\bar{\theta}|_{S_m} \in H^\infty |_{S_m}$, 
					\item $\bar{\phi} \in \theta H^\infty|_{S_m},$ 
					\item $\phi\theta^* |_{S_m} \in H^\infty |_{S_m}$.
				\end{enumerate}
				
			\end{theorem}
			\begin{proof}
				The proof of compactness also follows analogously. The BTTO $T_\phi^\theta$ is compact if and only if both $\tau_\phi$ and $h_{\phi\theta^*}$ are compact. But Corollaries \ref{2c1} and \ref{2c2} implies that $\tau_\phi$ is compact if and only if $\bar{\theta}|_{S_m} \in H^\infty |_{S_m}$ or $\bar{\phi} \in \theta H^\infty|_{S_m}$, and $h_{\phi\theta^*}$ is compact if and only if $\bar{\theta}|_{S_m} \in H^\infty |_{S_m}$ or $\phi\theta^* |_{S_m} \in H^\infty |_{S_m}$. Combining them, we reach the required conclusion.
			\end{proof}
			
			A function $\phi \in L^2$ is said to be \emph{symmetric} if $\phi^* = \phi$. We now present an observation regarding symmetric inner functions.
			
			\begin{corollary}
				When the inner function $\theta$ is symmetric, then both the small and big truncated Toeplitz operators are zero, finite rank, and compact under the same conditions, that is, 
				\begin{enumerate}[label=(\roman*)]
					\item  zero if and only if $\phi\in \vee\{\bar{\theta}\}$. 
					\vspace{0.1in}
					
					\item  finite in rank if and only if either $\theta$ is a finite Blaschke product or $\phi = \bar{\theta}(q_1+\bar{q_2}),$ where $q_1,q_2$ carries the same meaning as \ref{41} of Theorem \ref{4th1}.
					\vspace{0.1in}
					
					\item is compact if and only if for each support set $S_m$ either $\bar{\theta}|_{S_m}$ is constant or $\phi|_{S_m} = \bar{\theta}|_{S_m}.$
				\end{enumerate}
			\end{corollary}
			
			\begin{proof}
				If the inner function $\theta$ is symmetric, then, by Lemma~\ref{4l2}, the vanishing, finite-rank, and compactness properties of the BTTO $T_\phi^\theta$ reduce to the corresponding properties of both $\tau_\phi$ and $h_{\theta\phi}$. And these two operators are the same operators that determine the corresponding behavior of the STTO.
			\end{proof}

			\section{Remarks}
			\textbf{\emph{Special Remark:}} Consider the Hankel operator $\ch_\phi:H^2\rightarrow\ho$, defined by $\ch_\phi = QM_\phi|_{H^2}$. Now corresponding to the decompositions $H^2=\1\oplus\2$, and $\ho= \overline{z\1} \oplus \overline{z\2},$ we have the following block-matrix representation:
			$$ \scalebox{1.1}{$\ch_\phi :=
				\begin{array}{cc|c}
					\theta H^2 & K_\theta \\ \cline{1-2}
					\widetilde{\mathbb{H}}_\phi & \widehat{Z} & \overline{\theta H_0^2} \\
					Z & \Gamma_\phi^\theta & \overline{z\2} 
				\end{array}$}. $$
			
			In this study, the operator at the $(2,1)$ position of the matrix above yields another type of restricted operator. We call this operator a small restricted Hankel operator (SRHO) and define it as
			$$\widehat{\h}_\phi:\eta H^2 \rightarrow \overline{z\2} ~ \text{ by } ~ \widehat{\h}_\phi (f) = P_{\bar{\theta}}(\phi f),$$ 
			where $\eta,\theta$ are inner functions and $\phi \in \li.$
			
			Note that due to the unitary equivalence $\f P\f =Q$, the classical Hankel operators $\ch_\phi$ and $H_\phi$ share many key properties. On the other hand, the anti-unitary map $V$ on $L^2$ by $(Vf)(z)= \bar{z}\overline{f(z)}$ (as in \cite{DZ}), yields an equivalence between the projections $\pt$ and $P_{\bar{\theta}}$ because $$ V\pt V^{-1} = P_{\bar{\theta}},$$ which implies RHO $\hp$ and SRHO $\widehat{\hp}$ behave similarly in so many aspects, such as compactness of all kinds. 
			
			\begin{theorem}
				Let $\eta$ and $\theta$ be two inner functions and let $\phi\in \li$. The small restricted Hankel operator $\widehat{\hp}$ is 
				\begin{enumerate}[label=(\roman*)]
					\item zero if and only if $\phi\in \bar{\eta}H^\infty,$
					\item finite rank if and only if either $\theta$ is finite Blaschke product or $\phi \in \bar{\eta}(H^\infty + \mathcal{R})$, and
					\item  compact if and only if either $\bar{\theta}|_{S_m} \in H^\infty|_{S_m}$ or $\phi|_{S_m} \in \bar{\eta}H^\infty|_{S_m}.$
				\end{enumerate}
			\end{theorem}
			\begin{proof}
				One can prove the theorem either by utilizing the above anti-unitary relation or by following a similar argument to that done in Section 2. Here we provide a short outline of it.
				
				The orthogonal projection \begin{align*}
					& P_{\bar{\theta}}= Q - M_ {\bar{\theta}}QM_{\theta}\\
					& \implies P_{\bar{\theta}} |_{\ho} = I-dT_{\bar{\theta}} dT_{\theta} = \ch_{\bar{\theta}} \ch_{\bar{\theta}}^*,
				\end{align*}
				where $dT_{\phi}:\ho\rightarrow\ho$ is the dual Toeplitz operator (see \cite{DS}), defined by $dT_{\phi}(f)=Q(\phi f)$ for $\phi\in\li.$
				
				Now, using the same argument as RHO, we can show that $\widehat{\h_{\phi}}$ is zero, finite rank, or compact if and only if $\ch_{\bar{\theta}}\ch_{\eta \phi}$ is zero, finite rank, or compact, respectively. Therefore using the Lemmas \ref{2l4}, \ref{2l5}, and \ref{2l6}, we achieve our desired results.
			\end{proof}
			
			\textbf{\emph{Concluding Remark:}}  Throughout the article, we have considered that the symbols of RTO, RHO (and SRHO) are essentially bounded. However, they can be presented as densely defined operators, such as, for $\phi\in L^2$ and $\phi f \in \eta H^2 \cap\li$: 
			\begin{enumerate} [label=(\roman*)]
				\item $\tp: \eta H^2\rightarrow\2 ~\text{ by }~ \tp (f):= \pt (\phi f),$
				\item $\hp: \eta H^2\rightarrow\2 ~\text{ by }~ \hp (f):= \pt \f(\phi f), $
				\item $\widehat{\hp}: \eta H^2\rightarrow \overline{z\2} ~\text{ by }~ \widehat{\hp} (f):= P_{\bar{\theta}} (\phi f).$
			\end{enumerate}
			Now, we conclude the article with the following queries.
			
			\begin{enumerate}
				\item What are the conditions required to conclude that these restricted operators are bounded?
				
				\item Do we always have a bounded symbol for a bounded restricted operator?
			\end{enumerate}
			Note that these sorts of open problems for truncated operators (raised by Sarason in \cite{DES}) were addressed in \cite{BCFMT, BBK, GM, RL}.

			\section*{{{Acknowledgment}}}
			We thank Dr. Chandan Pradhan for reading the manuscript and making several helpful suggestions. P. Aroda is supported by NBHM Ph.D. fellowship, File No: 0203/13(18)/2021-R\&D-II/13138, by the Department of Atomic Energy, Government of India. A. Chattopadhyay is supported by the Core Research Grant (CRG), File No: CRG/2023/004826, by the Science and Engineering Research Board(SERB), Department of Science \& Technology (DST), Government of India. S. Jana gratefully acknowledges the support provided by IIT Guwahati, Government of India.

			\noindent $^{*}$ [ P. Aroda ] Department of Mathematics, Indian Institute of Technology Bombay, Mumbai, 400076, India.\\
			\textit{Email address:}~ priyanka.aroda@iitb.ac.in, rspriyanka0412@gmail.com.
			
			\noindent $^{**}$ [ A. Chattopadhyay ] Department of Mathematics, Indian Institute of Technology Guwahati, Guwahati, 781039, India.\\
			\textit{Email address:}~ arupchatt@iitg.ac.in, 2003arupchattopadhyay@gmail.com.

			\noindent $^{***}$ [ S. Jana ] Department of Mathematics, Indian Institute of Technology Guwahati, Guwahati, 781039, India.\\
			\textit{Email address:}~ supratimjana@iitg.ac.in, suprjan.math@gmail.com.
			
		\end{document}